\documentclass[a4paper,latexcad,twoside,11pt]{article}

\usepackage{authblk}

\usepackage{enumerate}

\usepackage[mathscr]{euscript}
\usepackage{xcolor,color}
\usepackage{amsthm}
\usepackage{amsmath}
\usepackage{stmaryrd} 
\usepackage{amssymb}
\usepackage{makeidx}
\usepackage{indentfirst}
\usepackage{amsfonts}
\usepackage{bm}
\usepackage{cases}
\usepackage{graphicx} 
\usepackage{subcaption}
\usepackage{hyperref}
\usepackage[a4paper]{geometry}

\newcommand\red[1] {{\color{red} #1}}
\newcommand\Red[1] 
{{\color{red} {\bf #1}}}

\newtheorem{theorem}{Theorem}
\newtheorem{lemma}[theorem]{Lemma}
\newtheorem{problem}{Problem}
\newtheorem{corollary}[theorem]{Corollary}

\newcommand{\brk}[1]{\llbracket #1 \rrbracket}

\newcommand\equ[2]
{
	\begin{equation}\label{#1}
		#2
	\end{equation}
}

\newcommand \eqn[2]
{\begin{eqnarray}\label{#1}
		#2
	\end{eqnarray}
}

\def \N {{\mathbb N}}
\def \hyh {{\cal H}}
\def \sets {{\cal S}}
\def \sete {{\cal E}}
\def \setm {{\cal M}}
\def \infe {\equiv_{\infty}} 
\def \infl {<_{\infty}}

\begin{document}
	
\baselineskip 5.7mm

\title{DP color functions versus chromatic polynomials for hypergraphs (I)}

\author[1]{Ruiyi Cui}
	
\author[1]{Liangxia Wan\thanks{ Corresponding author. 
Email: lxwan@bjtu.edu.cn. 
}}
	
\author[2]{ Fengming Dong\thanks{
Email: fengming.dong@nie.edu.sg and donggraph@163.com.}}

\affil[1]{\footnotesize 
	School of Mathematics and Statistics,
	Beijing Jiaotong University, Beijing 100044, China}

\affil[2]{\footnotesize National Institute of Education,
	Nanyang Technological University, Singapore}

\date{}

\maketitle

\begin{abstract}
For a hypergraph $\mathcal{H}$,
the DP color function $P_{DP}(\mathcal{H},k)$ 
of $\hyh$
is an extension of the chromatic polynomial $P(\mathcal{H},k)$ with the property that $P_{DP}(\mathcal{H},k) \le P(\mathcal{H},k)$ for all positive integers $k$. 
In this article, we primarily investigate the influence of the minimum cycle length on the DP-coloring function, as well as the relevant properties of the DP-coloring function of 
the join $\hyh \vee K_p$ of $\hyh$ and $K_p$.
We show that
for any linear and 
uniform hypergraph $\mathcal H$  
with even girth,
there exists 
a positive integer
$N$ such
that $P_{DP} (\mathcal H, k) < P(\mathcal H, k)$ 
for all integers $k\ge N$, 
and this conclusion also holds 
for any hypergraph $\mathcal{H}$
that contains an edge $e$ 
with the properties that  
$\hyh-e$ has exactly 
$|e|-1$ components and 
any shortest cycle in $\hyh$ containing $e$ is an even cycle.
For the hypergraph $\hyh\vee K_p$, 
we prove that if $\hyh$ is uniform,  
then there exist positive integers $p$ and $N$ such that $P_{DP}(\mathcal{H} \vee  K_p,k)=P(\mathcal{H} \vee K_p,k)$ holds for all integers $k\geq N$.

\end{abstract}

\section{Introduction}

For any graph $G$ and a positive integer $k$, a proper $k$-coloring of $G$
is a way of assigning 
a set of $k$ colors 
to vertices in $G$ 
such that no two adjacent vertices share the same color. 
In the 1970s, the vertex coloring was independently generalized to the list coloring by Vizing \cite{Vi76}  and Erd\H{o}s, Rubin, and Taylor \cite{ERT79},  respectively. 
Later, in 2015, Dvo\v{r}\'{a}k  and 
Postle \cite{DP18} generalized the list coloring to the DP-coloring (also called correspondence coloring) 
for proving the result 
that every planar graph without cycles of length $4$ to $8$ is $3$-choosable.  
The DP-coloring has been extensively studied \cite{BHKMMSTW22,DY22,KM21,LLYY19, MP22,Mu22}.  

Let $\N$ denote 
	the set of 
	positive integers  
	and let $\brk{k} $ denote the set  $\{i\in \N: i\le k\}$.
In 1912, Birkhoff \cite{Bi12} introduced the chromatic polynomial $P(G,k)$, 
which counts the number of $k$-proper colorings of $G$
for any $k\in \N$, 
for the purpose of proving the four-color conjecture. 
Kostochka and Sidorenko \cite{KS92} extended $P(G,k)$
to the list color function $P_l(G,k)$ in 1992. 
Some results on the list color functions
were provided by Thomassen \cite{Th09}.
 In 2021, 
Kaul and Mudrock~\cite{KM21}
further extended 
$P_l(G,k)$
to the DP color function $P_{DP}(G,k)$ of $G$.

 Obviously, 
 $P_l(G,k)\le P(G,k)$
 holds for any graph $G$ and 
 $k\in \N$.
 It has been shown that 
 for any graph $G$ of size $m$,  
 $P_l(G,k)=P(G,k)$
 holds for all integers $k\ge m-1$ (see \cite{Dong23}).
However, 
	there is no such a 
 conclusion on the relation between the  
  DP color function of a graph $G$ and its 
 chromatic polynomial. 
 Kaul and Mudrock~\cite{KM21} proved that if $G$ is a graph with even girth, then there exists an $N\in \mathbb N$ such
 that $P_{DP} (G, k) < P(G, k)$ whenever $k \ge N$. 
 An edge $e$ in a connected graph $G$ is called a bridge of $G$ if $G- e$ is disconnected. 
 For each edge $e$ in $G$, let $\ell(e) = \infty$ if $e$ is a bridge of $G$, and let $\ell(e)$ denote the length of a shortest cycle in $G$ containing $e$ otherwise.
 Dong and Yang \cite{DY22} show that if 
 $\ell(e)$ is even for some $e\in E(G)$, then 
 there exists an $N\in \mathbb N$ such that 
 $P_{DP}(G) < P(G)$ 
 for all $k\ge N$.
Mudrock and Thomason~\cite{Mu21} showed that for any graph $G$, 
there exists  
$N\in \N$ such that $P_{DP}(G\vee K_1, k)
=P_{DP}(G\vee K_1, k)$
for all $k\ge N$,
where $G\vee K_1$ 
	is the graph obtained from $G$ by adding a new vertex $w$ and adding new edges joining $w$ to all vertices in $G$.
Their result answered a problem 
posed by Kaul and Mudrock~\cite{KM21}.
All these results have been further extended by Zhang and Dong~\cite{Zh23}.

Helgason \cite{He06} introduced the chromatic polynomial 
$P(\hyh,k)$ of a hypergraph 
$\hyh$
in 1972 and 
the authors of this article introduced
	the DP color function 
	$P_{DP}(\hyh,k)$ 
	of a hypergraph $\hyh$ in \cite{CWD25}.
In this article, we partially extend the above comparison between 
$P_{DP}(G,k)$ and $P(G,k)$ for graphs $G$ to a corresponding comparison between 
$P_{DP}(\hyh,k)$ and $P(\hyh,k)$ 
for hypergraphs $\hyh$.

For a hypergraph $\hyh=(V,E)$,
a cycle $C$ of length $p$ in $\mathcal{H}$ is a subhypergraph consisting of $p$ distinct vertices $v_1, e_2, \dots, v_p$ and $p$ distinct edges $e_1, e_2, \ldots, e_p$ such that $\{v_{i-1}, v_{i}\}
\subseteq  e_i$ for each $i \in \brk{p}$ (indices are taken modulo $p$).
For any edge $e\in E$,
{\it the girth of $e$} in $\hyh$,
denoted by $\ell(e)$, 
is defined to be $\infty$
if $\hyh$ has no cycles containing $e$, 
and 
the length of a shortest cycle $C$ 
in $\hyh$ which contains $e$
otherwise. 
The girth of $\hyh$, 
denoted by $g(\hyh)$, 
is defined to be the minimum 
value of $\ell(e)$ over all edges 
$e$ in $\hyh$.

For any two functions $f(x)$
and $g(x)$, 
we may write $f(k)\infe g(k)$
if there exists $N\in \N$ such that $f(k)=g(k)$ holds
for all integers $k\ge N$;
and $f(k)\infl g(k)$
if there exists $N\in \N$ such that $f(k)<g(k)$ holds
for all integers $k\ge N$. We will study the following three
problems in this article.

 \begin{problem}\label{Pro-1}
 	Is it true that for any hypergraph $\mathcal H$,
 		if $g(\hyh)$ is even, 
 		then $P_{DP}(\hyh, k)\infl
 		P(\hyh, k)$?
 \end{problem}
 
  \begin{problem}\label{Pro-2}
 Is it true that for any hypergraph $\mathcal H$,
  	if $\ell(e)$ is even for some edge $e$ in $\hyh$,  
  	then $P_{DP}(\hyh, k)\infl
  	P(\hyh, k)$?
 \end{problem}

The {\it join} of two vertex-disjoint  hypergraphs $\hyh_1$ and $\hyh_2$,
denoted by 
$\mathcal{H}_1 \vee \mathcal{H}_2$, 
 is the hypergraph obtained from $\mathcal{H}_1$ and $\mathcal{H}_2$ by adding new edges $\{u,v\}$ for all vertices $u$ in $\mathcal{H}_1$ and all vertices $v$ in $\mathcal{H}_2$. 
 
  \begin{problem}\label{Pro-3}
  	 	Is it true that for any hypergraph $\mathcal H$,
  there exist $p\in \mathbb{N}$
  such that 
  $P_{DP}(\hyh\vee K_p, k)\infe
  P(\hyh\vee K_p, k)$?
 \end{problem}

Theorem~\ref{Gir-1} confirms that Problem~\ref{Pro-1} has a positive answer for all 
linear and uniform hypergraphs,
which Problem~\ref{Pro-2}
and Problem~\ref{Pro-3}
are partially answered by 
Theorems~\ref{EvenCyc}
and~\ref{Ans-3}, respectively.

\section{Background}

In this article each hypergraph is nonempty, finite and simple unless otherwise noted. 
A hypergraph $\mathcal{H} = (V, E)$ consists of a finite non-empty set $V$ and a subset $E$ of $2^V$ (i.e., the power set of $V$) 
with $|e| \geq 2$ for each $e \in E$. 
For any hypergraph $\hyh=(V, E)$, let $n(\mathcal{H})=|V|$ and  $m(\mathcal{H})=|E|$. 
A hypergraph $\mathcal{H}$ is \textit{$r$-uniform} if  $|e| = r$ for each edge $e \in E$. 
The \textit{degree} $d(v)$ of a vertex $v \in V$ in $\hyh$ is the number of edges in $\hyh$ containing $v$. A hypergraph is \textit{linear} if every pair of distinct edges intersects in at most one vertex. 
 A \textit{hypertree} is a connected linear hypergraph without cycles. A \textit{unicyclic hypergraph} is a connected hypergraph containing exactly one cycle. A hypergraph $\mathcal{H}= (V,E)$ is \textit{complete} if $E = P(V) \setminus \{ \emptyset \}$ where $P(V)$ denotes the powerset of $V$. A subset $S\subseteq E(\mathcal{H})$ is also called a {\it partial 
hypergraph} of $\mathcal{H}$.  

\subsection{The chromatic polynomial of a hypergraph}
 
For any positive integer $k$, a \textit{proper $k$-coloring} of the hypergraph $\mathcal{H}$ is a mapping $f:V(\mathcal{H})\to \brk{k}$ such that, for every edge $e\in E(\mathcal{H})$, there exist two distinct vertices $u,v\in e$ satisfying $f(u)\neq f(v)$. The number of proper $k$-colorings of $\mathcal{H}$ is a polynomial $P(\mathcal{H},k)$ in the variable $k$, of degree $|V(\mathcal{H})|$ in $k$, called \textit{chromatic polynomial} of $\mathcal{H}$.

Dohmen \cite{Do95} generalized the result obtained by Birkhoff and Whitney \cite{Wh32} as follows.
\begin{lemma}[\cite{Do95}] \label{dohmen}
	Let $\mathcal{H}$ be a hypergraph, 
	$k \in \mathbb{N}$. Then
	$$
		P(\mathcal{H},k) = \sum\limits_{S \subseteq \mathcal{H}} (-1)^{m(S)}k^{n(\mathcal{H})-n(S)+c(S)}
		$$
	where $c({S})$ is the number of components of ${S}$.
\end{lemma}

Let $\mathcal{H} = (V, E)$
and $V_0 \subset V$.
The \emph{contraction} of $\mathcal{H}$ onto $V_0$, denoted $\mathcal{H} \cdot V_0$, is obtained by identifying all vertices in $V_0$ into a single new vertex $v_0 \notin V$. Formally, $\mathcal{H} \cdot V_0$ is the hypergraph with vertex set $(V \setminus V_0) \cup \{v_0\}$ and edge set
\[
\bigl\{e \in E : e \cap V_0 = \emptyset \bigr\} \cup \bigl\{ 
(e \setminus V_0) \cup \{v_0\} : e \cap V_0 \neq \emptyset \bigr\}.
\]

Jones \cite{Jo78} generalized the deletion-contraction formula for chromatic polynomials of graphs to hypergraphs.

\begin{lemma}[\cite{Jo78}]
	Let $\mathcal{H} = (V, E)$ be a hypergraph. For any $e \in E$,
	$$
		P(\mathcal{H},k) = P(\mathcal{H}-e, k) - P(\mathcal{H}/e,k) \label{D-C}
		$$
	where $\hyh-e$ is the hypergraph $(V, E\setminus \{e\})$ and 
	$\mathcal{H}/e$ is obtained from $\mathcal{H}$ by contracting edge $e$ (i.e. the hypergraph $(\mathcal{H} - e) \cdot e$).
\end{lemma}

\subsection{The DP color function of a hypergraph}

 Bernshteyn and Kostochka \cite{BK19} introduced the DP-coloring of a hypergraph in 2019. Let $\varphi \colon X \rightharpoonup Y$ be a \textit{partial map} $\varphi$ from a subset of $X$ to $Y$. $[X \rightharpoonup Y]$ denotes the set of all partial maps. For $k \in \mathbb{N}^+$, consider a family $\mathcal{F} \subseteq [X \rightharpoonup \brk{k} ]$ of partial maps. The domain $\text{dom}(\mathcal{F})$ of $\mathcal{F}$ is defined by
$$
\text{dom}(\mathcal{F}) := \{ \text{dom}(\varphi) : \varphi \in \mathcal{F} \}.
$$

Then each $\varphi \in \mathcal{F}$ is a subset of the Cartesian product $X \times \brk{k} $. Thus we can view $\mathcal{F}$ as a hypergraph on $X \times \brk{k} $ and view $\text{dom}(\mathcal{F})$ as a hypergraph on $X$.
A function $f: X \to \brk{k} $ avoids a partial map $\varphi : X \rightharpoonup \brk{k} $ if $\varphi \nsubseteq f$. Given a family $\mathcal{F} \subseteq [X \rightharpoonup \brk{k} ]$ of partial maps, if a function $f : X \to \brk{k}$ avoids all $\varphi \in \mathcal{F}$, then it is a \textit{$(k, \mathcal{F})$-coloring} (or simply an \textit{$\mathcal{F}$-coloring}).

A \textit{$k$-fold cover} of $\mathcal{H}$ is a family of partial maps $\mathcal{F} \subseteq [X \rightharpoonup \brk{k}]$ such that $\text{dom}(\mathcal{F}) \subseteq V(\mathcal{H})$ and if $\text{dom}(\varphi) = \text{dom}(\psi)$ for distinct $\varphi, \psi \in \mathcal{F}$, then $\varphi \cap \psi = \emptyset$.
Let $\mathcal{F}_e = \{ \varphi \in \mathcal{F} : \text{dom}(\varphi) = e \}$ and let $\mathcal{F} \setminus \mathcal{F}_e = \{ \varphi \in \mathcal{F} : \text{dom}(\varphi) \neq e \}$. If $\mathcal{F}$ is a $k$-fold cover of $\mathcal{H}$, then $\mathcal{F} \setminus \mathcal{F}_e$ is a $k$-fold cover of $\mathcal{H} - e$. If $\mathcal{F}$ is a $k$-fold cover of a hypergraph $\mathcal{H}$, then for each edge $e$ in $\mathcal{H}$
$$
|\mathcal{F}_e| = |\{ \varphi \in \mathcal{F} : \text{dom}(\varphi) = e \}| \leq k.
$$
If $|\mathcal{F}_e| = k$ for each $e\in E(\mathcal{H})$, then $\mathcal F$ is perfect.

A hypergraph $\mathcal{H}$ on a set $X$ is \textit{$k$-DP-colorable} if each $k$-fold cover $\mathcal{F}$ of $\mathcal{H}$ admits a $(k, \mathcal{F})$-coloring $f: X \to \brk{k}$.
Let $\mathcal{H}$ be a hypergraph with $E(\mathcal{H})=\{e_j: j\in \brk{m}\}$. A \textit{natural $k$-cover} of $\mathcal{H}$ is denoted by
$$
\iota_{\mathcal H}(k)
=\left \{\varphi^{(j)}_i:  i\in \brk{k},  j\in \brk{m}\right \},
$$ 
where $\varphi^{(j)}_i$ is the mapping: $e_j\mapsto \{i\}$. 
If $\mathcal I$ is the
set of $\iota_{\mathcal H}(k)$-colorings of $\mathcal H$ and $\mathcal J$ is the set of proper $k$-colorings of $\mathcal H$, then the function
$f: \mathcal J\rightarrow \mathcal I$ given by
$$ f(c) = \{(v,c(v)): v\in V
(\mathcal H)\}$$
is a bijection between 
proper $k$-colorings
and 
$\iota_{\mathcal H}(k)$-colorings  of $\mathcal H$.

For a 
$k$-fold cover $\mathcal F$
of a hypergraph $\mathcal H$,
let $P_{DP}(\mathcal H,\mathcal F)$ denote the number of $\mathcal F$-colorings of $\mathcal H$. 
The \textit{DP color function}, denoted by $P_{DP}(\mathcal H,k)$,
 is the minimum value of $P_{DP}(\mathcal H,\mathcal F)$ over all $k$-fold covers $\mathcal{F}$ of $\mathcal{H}$.

It is clear that
for any hypergraph $\mathcal H$ and $k\in \mathbb{N}$
$$
	P_{DP}(\mathcal H,k)\le P(\mathcal H,k).
	$$

We end this section with two 
known results on upper and lower bounds of DP color functions of hypergraphs in \cite{CWD25}.

\begin{lemma}[\cite{CWD25}] \label{CWD}
	Let $\mathcal{H} = (X, E)$ is an r-uniform hypergraph with $n$ vertices and $m$ hyperedges, $n,m \in \mathbb{N}$. Then for each $k \in \mathbb{N}$
	$$
		P_{DP}(\mathcal{H},k) \le 
		k^{n-(r-1)m}(k^{r-1}-1)^{m}.
	$$
\end{lemma}

\begin{lemma}[\cite{CWD25}]
	\label{CWD-1}
	Suppose that $\mathcal{F}$ is a $k$-fold cover of a hypergraph $\mathcal{H}$, where $k \in \mathbb{N}$. 
	Suppose that 
	$e$ is an edge in $\hyh$ with 
	$|e|=n_{e} \ge 2$ and $c(\mathcal{H}-e)=n_e-1$.
	Then $\mathcal{F}' = \mathcal{F} - \mathcal{F}_{e}$ is a $k$-fold cover of $\mathcal{H}-e$, where $ \mathcal{F}_{e} = \{ \varphi \in \mathcal{F} : dom(\varphi) = e\}$. If there is a bijection between $\mathcal{F}'$-colorings of $\mathcal{H}-e$ and proper $k$-colorings of  $\mathcal{H}-e$, then there exists a $k$-fold cover $\mathcal{F}^{*}$ of $\mathcal{H}$ such that
	\begin{eqnarray*}
		\nonumber
		P_{DP}(\mathcal{H},\mathcal{F}) &\ge& P_{DP}(\mathcal{H}, \mathcal{F}^{*})\\
	&	= &\min \left\{ P(\mathcal{H},k), \frac{(k^{n_e-1}-1)P(\mathcal{H}-e,k)-k^{n_e-2}P(\mathcal{H},k)}{k^{n_e-2}(k-1)} \right\}. 		
	\end{eqnarray*}	
\end{lemma}

\section{Impact by even 
	shortest cycles}

 In this section, we consider the impact of the minimum cycle length and even cycle length on the DP color function of a hypergraph. 
 We will show that for a hypergraph $\hyh$, 
 $P_{DP}(\hyh,k)\infl P(\hyh,k)$ 
 holds under any of the following conditions:
 \begin{enumerate}[(i)]
 	\item $\hyh$ is uniform and  
 	linear, and $g(\hyh)$ is even; or 
 	
 	\item there exists an edge $e$ in $\hyh$ such that $c(H-e)=|e|-1$ and $\ell(e)$ is even. 
 \end{enumerate}

\subsection{When $g(\hyh)$ is even}

Let's first establish
the following result 
by applying Lemma \ref{dohmen}.

\begin{lemma}\label{1-1}
	Let $ \mathcal{H}=(V,E)$ be a connected linear $r$-uniform hypergraph with $|V|=n$ and $|E|=m$.
Let $g(\hyh)=z$ and let $t$ be the number of cycles of length $z$ in $\mathcal{H}$. 
	Then 
		\equ{eq1}
	{
	P(\mathcal{H}, k) =f(k)+
	(-1)^{z}tk^{n-z(r-1)+1}
	+ \sum_{i=0}^{z-1}(-1)^{i}
	{m \choose i}
	k^{n-i(r-1)},
}
where 
$f(k)$ is a polynomial in $k$ 
of degree $n-z(r-1)$.
\end{lemma}

\begin{proof}
Let $\sets_0$ be the set of $S\subseteq \hyh$ with $m(S)\le z-1$, 
$\sets_1$ be the set of $S\subseteq \hyh$ 
with $m(S)=z$ and $S$ is a cycle,
and $\sets_2$ be the set of 
$S\subseteq \hyh$ 
with $S\notin \sets_0\cup \sets_1$.	
	By Lemma \ref{dohmen}, we know that
	\eqn{eq2}
	{
		P(\mathcal{H},k) 
		= \sum\limits_{S \subseteq \mathcal{H}} (-1)^{m(S)}k^{n(\mathcal{H})-n(S)+c(S)}
		=\sum_{s=0}^2
	\sum\limits_{S\in \sets_s} (-1)^{m(S)}k^{n(\mathcal{H})-n(S)+c(S)}.
	}
For any $S\in \sets_0$,
we have $m(S)<z$, implying that 
$S$ is acyclic.
Assume that  $m(S) =i$.
Then $0\le i \le z-1$
and $n(S)=(r-1)i+c(S)$, implying that 
	$$
	n(\hyh)-n(S)+c(S)=n-i(r-1).
	$$
	Clearly, 
	for $0\le i\le z-1$, 
	there are exactly ${m\choose i}$ elements $S\in \sets_0$ 
 with $m(S)=i$.
	
By assumption, $|\sets_1|=t$.	For any $S\in \sets_1$, 
	$m(S)=z$ and $S$ is a cycle,
	and thus
	 $n(S)=(r-1)z$ and $c(S)=1$, implying that 
	   	$$
	   n(\hyh)-n(S)+c(S)=n-(r-1)z +1.
	   $$

For any $S\in \sets_2$, 
if $m(S)=z$, then 
$S$ is not a cycle,
and thus  $n(S)=c(S)+(r-1)z$,
implying that 
$$
 n(\hyh)-n(S)+c(S)
 =n-(r-1)z.
$$

For any $S\in \sets_2$, 
if $m(S)\ne z$,
then $m(S)>z$.
In this case, it can be shown that 
$n(S)\ge c(S)+(r-1)z+1$,
implying that 
$$
n(\hyh)-n(S)+c(S)
\le n-(r-1)z-1.
$$
Thus, $f(k)=
\sum\limits_{S\in \sets_2} (-1)^{m(S)}k^{n(\mathcal{H})-n(S)+c(S)}
$ is a polynomial of degree $n-(r-1)z$.

Hence 
 (\ref{eq1}) follows directly from (\ref{eq2}).
\end{proof}

By applying Lemma~\ref{1-1}, 
we establish the following result, which partially answers Problem \ref{Pro-1}.

\begin{theorem}\label{Gir-1}
Let $ \mathcal{H}$ be an $r$-uniform linear hypergraph.
If $g(\hyh)$ is even,
then $P_{DP}(\mathcal{H},k) 
\infl P(\mathcal{H},k)$.
\end{theorem}

\begin{proof}
Without loss of generality, let $\mathcal{H}$ be a connected hypergraph with $n$ vertices and $m$ edges. By Lemma \ref{CWD}, we know that
\equ{eq3}
{
P(\mathcal{H},k) - P_{DP}(\mathcal{H},k)
\ge 
	P(\mathcal{H},k) -
	k^{n-(r-1)m}(k^{r-1}-1)^{m}.
}
Suppose that 
$$
P(\mathcal{H}, k) = \sum^{m}_{i=0}
(-1)^{i}a_{i}k^{n-i(r-1)}
$$ 
and $t$ is the number of cycles
in $\hyh$ of length $z$.
Let $z=g(\hyh)$.
Applying Lemma \ref{1-1} and the binomial theorem, 
\begin{eqnarray}\label{eq3}
& & P(\mathcal{H},k) - k^{n-(r-1)m}(k^{r-1}-1)^{m}
	\nonumber\\ 
	& = & \sum_{i=0}^{z-1}(-1)^{i}a_{i}k^{n-i(r-1)}+(-1)^{z}tk^{n-z(r-1)+1}+f(x)-\sum_{i=0}^{m}(-1)^{i}\binom{m}{i} k^{n-i(r-1)} 
	\nonumber \\
	& = & (-1)^{z}tk^{n-z(r-1)+1} + f(x) -\sum_{i=z}^{m}(-1)^{i}\binom{m}{i}m^{n-i(r-1)},
\end{eqnarray}
where $f(k)$ is introduced in 
Lemma~\ref{1-1}.
Clearly, $tk^{n-z(r-1)+1}$ is the dominant of the last expression 
in (\ref{eq3}).
Since $z$ is even, $(-1)^z=1$
and thus 
there exists $N \in \mathbb{N}$ such that 
for all $k\ge N$, 
$$P(\mathcal{H},k) - k^{n-m(r-1)}(k^{r-1}-1)^{m} > 0.$$ 
Hence, by (\ref{eq3}),
$ P(\mathcal{H},k) - P_{DP}(\mathcal{H},k)>0 
$ holds for all  $k\ge N$.
\end{proof}

\subsection{$\ell(e)$ is even 
for some edge $e$}

\begin{theorem} \label{prop1.1}
	Let $\mathcal H$ be a hypergraph. 
If there exists 
	an edge $e$ 
	in $\hyh$ with $|e|=n_e\ge 2$
	such that 
$c(\mathcal{H}-e)=n_e-1$ and  
	 $P(\mathcal{H}-e, k) < \frac{k^{n_e-1}}{k^{n_e-1}-1} P(\mathcal{H},k)$,
  then  
		$$
	P_{DP}(\mathcal H,k)<P(\mathcal H,k).
	$$

\end{theorem}
\begin{proof}
	Let ${\cal H}'={\mathcal H}-e$. There is a bijection between the $ \mathcal {F}'$-coloring  of $\mathcal {H}'$ and its proper $k$-coloring.
	So, Lemma \ref{CWD-1} implies that there exists an $k$-fold cover $\mathcal F$ of $\mathcal H$ such that
\equ{eq4}
{
P_{DP}(\mathcal H, \mathcal F)=\min \left\{ P(\mathcal{H},k), \frac{(k^{n_e-1}-1)P(\mathcal{H}-e,k)-k^{n_e-2}P(\mathcal{H},k)}{k^{n_e-2}(k-1)} \right\}.
}
	Since 
	$$
		P(\mathcal{H}-e, k) < \frac{k^{n_e-1}}{k^{n_e-1}-1} P(\mathcal{H},k),
		$$
	it follows that 
	\begin{align}			\label{equ:2-2}
	\frac{(k^{n_e-1}-1)P(\mathcal{H}-e,k)-k^{n_e-2}P(\mathcal{H},k)}{k^{n_e-2}(k-1)}<P(\mathcal H,k).	
	\end{align}
	By combining 
	(\ref{eq4}) and (\ref{equ:2-2}),
	we have
	$$
	P_{DP}(\mathcal H,k) \leq P_{DP}(\mathcal H,\mathcal F)=\frac{(k^{n_e-1}-1)P(\mathcal{H}-e,k)-k^{n_e-2}P(\mathcal{H},k)}{k^{n_e-2}(k-1)}<P(\mathcal H,k).
$$
	Thus the result is deduced.
\end{proof}

In order to prove the following theorem, which is a partial answer to Problem \ref{Pro-3}, we conclude the following result.

\begin{lemma}\label{lemma 9}
	Suppose that $\mathcal{H}=(V,E)$ is a  hypergraph
	and $e=\{v_i: i \in \brk{n_e}\}$ is an edge in $\mathcal{H}$ with $c(\mathcal{H}-e) = n_e-1$. Then the leading term of the polynomial $P(\mathcal{H}-e,k) -\frac{k^{n_e-1}}{k^{n_e-1}-1}P(\mathcal{H}, k)$ is 
	 $\frac{1}{k^{n_e-1}-1} \sum_{S \in {\mathcal E} }(-1)^{m(S)}k^{n(\mathcal{H})-n(S)+c(S)}$, 
	where $\mathcal E$ is the set of subsets $S$ of $E\setminus \{e\}$ such that  
	$v_1$ and $v_2$ are in the same component of 
	the spanning subhypergraph $\hyh\langle S \rangle$  of $\mathcal{H}$.
\end{lemma}
\begin{proof}
	 Suppose that $v_1, v_2 \in e$ are in the same component of $\mathcal{H}-e$.
	 Then $\sete\ne \emptyset$.
	 
	 For any $S\subseteq E\setminus \{e\}$, 
	 let 
	$\zeta(S) =| V(S) \cap e|$ and $\eta(S) = |V(S) \cap \{v_1, v_2\}|$. 
	If either $\zeta(S) < 2$ or 
	$\zeta(S) \ge 2$ and $\eta(S) \le 1$, then 
	$$n_{\mathcal{H}}(S)=n_{\mathcal{H}/e}(S)\mbox{ and }  c_{\mathcal{H}}(S)=c_{\mathcal{H}/e}(S).$$
	
	 By the assumption,  
	 $\mathcal E$ be the set of all subsets $S$ of $E \setminus \{e\}$ such that  $v_1,v_2$ are in the same component of the spanning subhypergraph $\hyh\langle S \rangle$  of $\mathcal{H}$. 
	 Let $\mathcal E'$ be the set $S$ of subsets of $E \setminus \{e\}$ with $S \notin \mathcal E $. 
	 For any $S\subseteq E\setminus \{e\}$, 
	 if $\zeta(S) \ge 2$ and 
	 $\eta (S)= 2$, then
	$$\left\{
	\begin{aligned}
	   &n_{\mathcal{H}}(S)=n_{\mathcal{H}/e}(S)+\zeta(S) -2,\  c_{\mathcal{H}}(S)=c_{\mathcal{H}/e}(S)+\zeta(S) -1 , \qquad S \in \mathcal E; \\ &n_{\mathcal{H}}(S)=n_{\mathcal{H}/e}(S)+\zeta(S) -1,\  c_{\mathcal{H}}(S)=c_{\mathcal{H}/e}(S)+\zeta(S) -1 , \qquad S \in \mathcal E'.
	\end{aligned}
	\right.
	$$
	Thus
    \begin{align}
    	\nonumber
    	P(\mathcal{H}-e,k) -\frac{k^{n_e-1}}{k^{n_e-1}-1}P(\mathcal{H}, k) &= \frac{1}{k^{n_e-1}-1}(k^{n_e-1}P(\mathcal{H}/e, k) - P(\mathcal{H}-e,k)) \\
    	\label{cycle-subtraction} \nonumber
    	&=\frac{1}{k^{n_e-1}-1} \sum_{S \in \sete}(-1)^{m(S)}
    	k^{n(\mathcal{H})-n(S)+c(S)}.
    \end{align}    	
\end{proof}

\begin{theorem}\label{EvenCyc}
	Suppose that $\mathcal{H}=(V,E)$ is hypergraph. 
If  there exists $e\in E$ with $|e|=n_e\ge 2$
	such that $c(\mathcal{H}-e) = 
	n_e-1$ and $\ell(e)$ is even, 
then $P_{DP}(\mathcal{H},k)\infl P(\mathcal{H},k)$.
\end{theorem}

\begin{proof} 
	Let $v_1$ and $v_2$ be distinct 
	vertices in $e$,
	and let $\sete$ be the 
	the set of subsets $S$ of $E\setminus \{e\}$ such that  $v_1,v_2$ are in the same component of the spanning subhypergraph $\hyh\langle S \rangle$  of $\mathcal{H}$.
	Let
	\begin{align}
		\nonumber
	A=\frac{1}{k^{n_e-1}-1}\sum_{S \in \mathcal E}(-1)^{m(S)}k^{n(\mathcal{H})-n(S)+c(S)}.
	\end{align}

	When $c(S)$ increases by $1$, $c(S)-n(S)$ increases by at least $a$ for some $a \geq 0$. Thus suppose that $S$ is connected and that $m(S)$ is minimum. 
	For any $S \in \mathcal E$, 
	$S\cup \{e\}$ contains 
	a cycle $C$ with $e\in C$,
	implying that $m(S)\ge \ell(e)-1$ 
	and $\min\limits_{S\in \sete}m(S)=\ell(e)-1$.
Since $\ell(e)$ is even, 
	$(-1)^{m(S)}=-1$ whenever 
	$m(S)=\ell(e)-1$, implying hat 
	$A<0$ when $k$ is sufficiently large.

	By Lamma \ref{lemma 9},
	$A$ is the leading term in the polynomial $P(\mathcal{H}-e, k) - \frac{k^{n_e-1}}{k^{n_e-1}-1} P(\mathcal{H},k).$	
	Then there exists $N \in \mathbb{N}$ such that for all integer $k \ge N$
	 $$P(\mathcal{H}-e, k) < \frac{k^{n_e-1}}{k^{n_e-1}-1} P(\mathcal{H},k).$$   The result can be proved by Proposition \ref{prop1.1}.
\end{proof}

\section{DP color function of $\mathcal{H} \vee K_p$} 

In this section we study the question whether taking the join of an arbitrary hypergraph with an appropriate clique makes the chromatic polynomial equal to the DP color function. Theorem \ref{Ans-3} is a partial answer to Problem \ref{Pro-3}.

It is obvious that for any hypergraph $\mathcal{H}$ 
and positive integers $p$ and $k$,
$$
P(\mathcal{H} \vee K_p, k) = \left (\prod_{i=0}^{p-1}(k-i)\right ) P(\mathcal{H}, k-p).
$$ 

The \emph{coloring number} of a hypergraph $\mathcal{H}$, denoetd by col($\mathcal{H}$), is the smallest integer $d$ for which there exists an ordering, $v_1, v_2, ..., v_n$, of the elements in $V(\mathcal{H})$ such that each vertex $v_i$ has at most $d-1$ neighbors among  $v_1, v_2, ..., v_{i-1}$ for $2\le i\le n$. Throughout this Section, assume $\mathcal{H}$ is a hypergraph with ${\rm col}(\mathcal{H}) = d \ge 3$ and that each cover is perfect unless otherwise noted.
If a cover $\mathcal F$ is not a perfect, then ${\mathcal F}_1$ is a perfect cover by adding partial maps from $\mathcal F$. It is obvious that $P_{DP}(\mathcal H,{\mathcal F}_1)\le P_{DP}(\mathcal H,\mathcal F). $

Let $\mathcal{H}=(V,E)$ be a hypergraph with 
$V=\{v_i:i\in \brk{n}\}$
and $E=\{e_{n+i}: i\in \brk{m}\}$.
Let $\setm= K_1 \vee \mathcal{H}$.
Then, $\setm$ is the hypergraph 
with vertex set $V\cup \{w\}$ 
and edge set 
$E\cup \{e_i: i\in \brk{n}\}$,
where $e_i=\{w,v_i\}$.
Consider a $k$-fold cover $$\mathcal F=\left \{
\varphi^{(i)}_j(e_i),\psi^{(i)}_j(e_{n+p}):\varphi^{(i)}_j(w)=j,i\in \brk n,p\in\brk m,j\in\brk k
\right \}.
$$ 
Let $\varPhi_j=\{ \varphi^{(i)}_j(w): i\in \brk{n}\}$ for each $j\in \brk k$. We say that $\varPhi_j$ is a \emph{level mapping}  for some $j \in \brk{k}$ if for any given $p \in \brk{m}$ there exists $q_p \in \brk{k} $ such that for each $v_i\in e_{n+p}$ $$\psi^{(p)}_{q_p}(v_i)
=\varphi^{(i)}_j(v_i).$$
For each $j\in \brk k$ let
$$
\varPsi_j = \left\{ \psi^{(p)}_{q}(e_{n+p}): \forall p \in \brk{m}, \exists q\in\brk{k} \text{ with } \psi^{(p)}_{q}(v_i)=\varphi^{(i)}_j(v_i) \right\}
$$
and let 
$$	
\mathcal{F}_j = \mathcal{F}\setminus ( \varPsi_j \cup\{\varphi^{(i)}_j(e_i):
i\in\brk n\}). 
\label{F_j}
$$

For example, consider the $3$-fold cover of $ \mathcal{H}
\vee K_1$ shown in Table \ref{tab:example} , where $\varPhi_2$ is a level mapping, but $\varPhi_1$ and $\varPhi_3$ are not level mappings.

\begin{table}[ht]
	\centering
	
	\begin{minipage}{0.5\textwidth}		
		\begin{subfigure}[t]{\textwidth}
			\centering
	\includegraphics[width=0.6\textwidth]{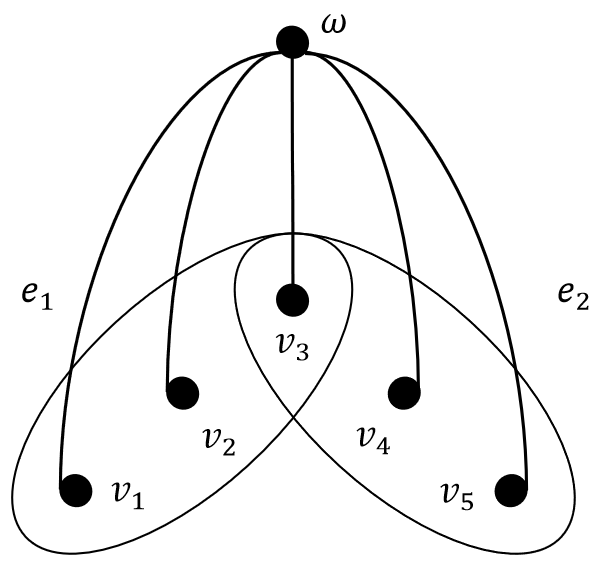}
			\caption{ $K_1 \vee \mathcal{H}$}
		\end{subfigure}
		
		\vspace{10pt}

		\begin{subfigure}[t]{\textwidth}
			\centering
			\setlength{\tabcolsep}{6pt} 
			\renewcommand{\arraystretch}{1.1} 
			\begin{tabular}{c c c c c c c}
				\hline
				& $w$ & $v_1$ & $v_2$ & $v_3$ & $v_4$ & $v_5$ \\
				\hline
				$\varphi^{(1)}_1$ & 1 & 1 &  &  &  & \\
				$\varphi^{(1)}_2$ & 2 & 2 &  &  &  & \\
				$\varphi^{(1)}_3$ & 3 & 3 &  &  &  & \\
				\hline
				$\varphi^{(2)}_1$ & 1 &  & 2 &  &  & \\
				$\varphi^{(2)}_2$ & 2 &  & 3 &  &  & \\
				$\varphi^{(2)}_3$ & 3 &  & 1 &  &  & \\
				\hline
			\end{tabular}
			
		\end{subfigure}
		\end{minipage}
		\hfill 
		\begin{minipage}{0.48\textwidth}
		\centering
		\setlength{\tabcolsep}{6pt}
		\renewcommand{\arraystretch}{1.1}
		\begin{tabular}{c c c c c c c}
			\hline
			& $w$ & $v_1$ & $v_2$ & $v_3$ & $v_4$ & $v_5$ \\
			\hline
			$\varphi^{(3)}_1$ & 1 &  &  & 3 &  & \\
			$\varphi^{(3)}_2$ & 2 &  &  & 1 &  & \\
			$\varphi^{(3)}_3$ & 3 &  &  & 2 &  & \\
			\hline
			$\varphi^{(4)}_1$ & 1 &  &  &  & 1 & \\
			$\varphi^{(4)}_2$ & 2 &  &  &  & 2 & \\
			$\varphi^{(4)}_3$ & 3 &  &  &  & 3 & \\
			\hline
			$\varphi^{(5)}_1$ & 1 &  &  &  &  & 1 \\
			$\varphi^{(5)}_2$ & 2 &  &  &  &  & 3 \\
			$\varphi^{(5)}_3$ & 3 &  &  &  &  & 2 \\
			\hline
			$\psi^{(1)}_1$ & & 1 & 2 & 3 &  &   \\
			$\psi^{(1)}_2$ & & 2 & 3 & 1 &  &   \\
			$\psi^{(1)}_3$ & & 3 & 1 & 2 &  &   \\
			\hline
			$\psi^{(2)}_1$ &  &  &  & 1 & 2 & 3  \\
			$\psi^{(2)}_2$ &  &  &  & 3 & 1 & 2  \\
			$\psi^{(2)}_3$ &  &  &  & 2 & 3 & 1  \\
			\hline
		\end{tabular}
		\end{minipage}
	
	\caption{$\mathcal{H}\vee K_1$ and a $3$-fold cover} 
	\label{tab:example}
\end{table}

\begin{lemma}\label{2-1}
	Let $\mathcal{F}$ is a $k$-fold cover of a hypergraph $\mathcal H$ and let $\varPhi=\{\varPhi_j:j\in\brk k\}$, 
	where $\varPhi_j=\{ \varphi^{(i)}_j(w): i\in \brk{n}\}$
	for each $j\in \brk{k}$.
	If $\varPhi$ has at least $k-1$ level mappings, then there is a natural bijection between the $\mathcal{F}$-colorings of $\mathcal M$ and the proper $k$-colorings of $\mathcal M$. Consequently $P_{DP}(\mathcal M, \mathcal{F}) = P(\mathcal M, k)$.
\end{lemma} 
\begin{proof}
	Without loss of generality, we suppose  that $\varPhi_j$ are 
	level mappings for 
	each $j\in \brk{k-1}$. 
	It is clear that $\varPhi_k$ is also a level mapping. There exists a bijection between $\mathcal{F}$ and the natural $k$-cover 
	$
	\iota_{\mathcal M}(k)$ by putting for every $i\in\brk n$, $p\in \brk m$
	$$\varphi^{(i)}_j(v_i)=j, \psi^{(p)}_q(v)=q \mbox{ for each }j,p\in \brk k, v\in e_{n+p}.$$
	Because 
	there is a natural bijection between the $\iota_{\mathcal M}(k)$-colorings of $\mathcal M$ and the proper $k$-colorings of $\mathcal M$, there is a natural bijection between the $\mathcal{F}$-colorings of $\mathcal M$ and the proper $k$-colorings of $\mathcal M$. Thus $$P_{DP}(\mathcal M, \mathcal{F}) = P(\mathcal M, k).$$
\end{proof}

\begin{lemma}\label{2-2}
	For a hypergraph
	$\mathcal{H}$ with ${\rm col}(\mathcal{H}) =d$, 
	let $\mathcal{F}$ is a $k$-fold cover of $\mathcal{H}$ and let $\varPhi=\{\varPhi_j:j\in\brk k\}$,
	where $\varPhi_j=\{ \varphi^{(i)}_j(w): i\in \brk{n}\}$
	for each $j\in \brk{k}$.
	If $\varPhi$ contains $s$ mappings that are not level mappings, there exists $e \in E(\mathcal{H})$ such that
	\begin{align}\nonumber
	P_{DP}(\mathcal{M},\mathcal{F}) \geq kP_{DP}(\mathcal{H},k-1)
	+s(k-d-|e|)^{n-|e|}.
	\end{align}
\end{lemma}
\begin{proof}
    Without loss of generality, suppose that $\varPhi_j$ is not a level mapping for each $j \in \brk{s}$. It is easy to know that
	\begin{align}\nonumber
	P_{DP}(\mathcal M,\mathcal{F})=\sum_{j = 1}^{k}P_{DP}(\mathcal{H},\mathcal{F}_j)\geq k P_{DP}(\mathcal{H},k-1)
	\end{align}
	where 
	$\mathcal{F}_j$ is defined 
		in Page ~\pageref{F_j}.

Let $j \in \brk{s}$.
Then $\varPhi_j$ is not a level mapping and 
	there exists $e_{n+p} \in |E(\mathcal{H})|$ such that 
	$$|\psi:{\rm dom}(\psi)=e_{n+p}, \psi\in\mathcal{F}_j|=l \le k-2.$$  
	For each $v_i\in e_{n+p}$, let 
	$$C_i=\brk{k}\setminus(\{\psi^{(p)}(v_i):\psi^{(p)}\in \mathcal{F}_j\}\cup\{\varphi^i_t(v_i)\}).$$
It is clear that for each $v_i\in e_{n+p}$
$$|C_i|=k-1-l.$$
Obviously, we can add $k-l-1$ partial maps with the domain $e_{n+p}$ and denote the set of these partial maps by $\varTheta$. 	
	Let $\mathcal{F}_j^*=\mathcal{F}_j\cup \varTheta$. Thus, $P_{DP}(\mathcal{H}, \mathcal{F}^{*}_j)$ is the number of $\mathcal{F}_j$-colorings of $\mathcal{H}$ avoiding maps in $\varTheta$. Hence, there are at least $P_{DP}(\mathcal{H},k-1)$ $\mathcal{F}_t$-colorings of $\mathcal{H}$ avoiding $\varTheta$.
	
	Suppose that $v_1, v_2, ..., v_n$ is an ordering of the vertices of $\mathcal{H}$ such that $v_i$ has at most $d-1$ neighbors preceding it in the ordering. Consider the following ordering $S$ of the vertices of $\mathcal{H}$ obtained by moving the vertices in $e_{n+p}$ to the front of the entire sequence:
	$$v_{j_1}, v_{j_2}, ..., v_{j_n}.$$
	Thus, there are at least $((k-1)-(d+|e|-1))^{n-|e|}
	=(k-d-|e|)^{n-|e|} \mathcal{F}_j$-colorings of $\mathcal{H}$ that can color $v_i \in V(e_{n+p})$ with $\psi\in\varTheta$. This immediately implies 
	$$P_{DP}(\mathcal{H}, \mathcal{F}_{j}) \ge P_{DP}(\mathcal{H},k-1)+(k-d-n_e)^{n-n_e}.$$

\end{proof}

It follows readily from Lemmas \ref{2-1} and \ref{2-2} that:
\begin{theorem}\label{th-2-1}
	Suppose that $\mathcal{H}$ is a hypergraph with ${\rm col}(G)\geq 3$. 
Let $n_0$ be the maximum value 
of $|e|$ over all edges in $\hyh$.
	Then for $k \geq {\rm col}(G)+n_0$,
	\begin{align}\nonumber
		P_{DP}(\hyh \vee K_1,k)\geq \min
		\left \{
		P(\mathcal{H}\vee K_1,k),kP_{DP}(\mathcal{H},k - 1)+f(x)\right \}
	\end{align}
	where $f(x) \geq 0$. In particular, when $\mathcal{H}$ is an r-uniform hypergraph,
	\begin{align}\nonumber
		P_{DP}(\mathcal{H}\vee K_1,k)\geq \min\left \{P(\mathcal{H}\vee K_1,k),kP_{DP}(\mathcal{H},k - 1)+2(k {\rm col}(\mathcal{H})-r)^{|V(G)| - r}\right \}.
	\end{align} 
\end{theorem}

\begin{corollary} \label{co-2-1}
Suppose that $\mathcal{H}$ is an $r$-uniform hypergraph with $n$ vertices and 
${\rm col}(\mathcal{H})\geq 3$. Then for any $p\in\mathbb{N}$ and $k \geq {\rm col}(\mathcal{H})+ r + p$,
\begin{align}\nonumber
	P_{DP}( \mathcal{H}\vee K_p,k)\geq \min\left\{P( \mathcal{H}\vee K_p,k),\left(\prod_{j = 0}^{p - 1}(k - j)\right)P_{DP}(\mathcal{H},k - p)+f(k)\right\}
\end{align}
where $f(k)$ is a polynomial in $k$ of degree $n - 3 + p$ with a leading coefficient of $2p$.
\end{corollary}
\begin{proof}
	The proof proceeds by induction on $p$. The result holds for $p=1$ by Theorem \ref{th-2-1}. Assume  that the result holds for all numbers less than $p$ ($p \geq 2$). Next consider the case $p$.
	
	Suppose that $k\geq \text{col}(\mathcal{H})+r + p$. Since $K_p \vee \mathcal{H}=K_1 \vee (K_{p - 1} \vee  \mathcal{H})$ and $\text{col}(K_{p - 1} \vee \mathcal{H})\leq \text{col}(\mathcal{H})+p - 1$, lemmas \ref{2-1} and \ref{2-2} imply that 
\begin{align}\nonumber
		P_{DP}( \mathcal{H}\vee K_p,k)&\geq \min\{P(\mathcal{H}\vee K_p,k),kP_{DP}(K_{p - 1}\vee \mathcal{H},k - 1)\\ \nonumber
		&\quad+2(k-\text{col}(K_{p - 1}\vee \mathcal{H})-n_e)^{n-n_e-1 + p}\}
\end{align}
	where $n_e$ denotes the length of a certain edge in $ K_{p-1} \vee \mathcal{H} $ and $n_e \le r$. Then $$(k-\text{col}(K_{p - 1}+\mathcal{H})-n_e)^{n-n_e-1 + p} \geq (k-\text{col}(K_{p - 1}\vee \mathcal{H})-r)^{n-r + p}$$
	 and
	\begin{align}
		\nonumber
		P_{DP}(K_p \vee \mathcal{H},k)&\geq \min\{P(K_p \vee \mathcal{H},k),kP_{DP}(K_{p - 1}\vee \mathcal{H},k - 1)\\ \nonumber
		&\quad+2(k-\text{col}(K_{p - 1}\vee \mathcal{H})-r)^{n- r + p}\}.
	\end{align}

	Because $k -1 \geq \text{col}(\mathcal{H})+r-1 + p $, by the inductive hypothesis it implies that 
	\begin{align}
		\nonumber
		P_{DP}&(K_{p - 1}\vee \mathcal{H},k - 1)\\ \nonumber
		&\geq \min\left\{P(K_{p - 1}\vee \mathcal{H},k - 1),\left(\prod_{j = 1}^{p - 1}(k - j)\right)P_{DP}(\mathcal{H},k - p)+f(k)\right\}
	\end{align}
	where $f(k)$ is a polynomial in $k$ of degree $n - r-2 + p$ with a leading coefficient of $2(p - 1)$.
		
	If $P(K_{p - 1}\vee \mathcal{H},k - 1)\leq\left(\prod_{j = 1}^{p - 1}(k - j)\right)P_{DP}(\mathcal{H},k - p)+f(k)$, then we get that 
	$$P_{DP}(K_{p - 1}\vee \mathcal{H},k - 1)=P(K_{p - 1}\vee \mathcal{H},k - 1).$$
	 This implies that 
	 $$kP_{DP}(K_{p - 1}\vee \mathcal{H},k - 1)=kP(K_{p - 1}\vee \mathcal{H},k - 1)=P(K_{p}+\mathcal{H},k).$$ Thus 
	\begin{align}\nonumber
			P_{DP}(K_{p}\vee \mathcal{H},k)&\geq\min\{P(K_{p}+\mathcal{H},k),kP_{DP}(K_{p - 1}\vee \mathcal{H},k - 1)\\ \nonumber
			&\quad+2(k - \text{col}(K_{p - 1}+\mathcal{H})-r)^{n-r-1 + p}\}.
	\end{align}
This implies $P_{DP}(K_{p}\vee \mathcal{H},k)=P(K_{p}+\mathcal{H},k)$ and the result is deduced. 
	
	Otherwise 
	$$P(K_{p - 1}\vee \mathcal{H},k - 1)>\left(\prod_{j = 1}^{p - 1}(k - j)\right)P_{DP}(\mathcal{H},k - p)+f(k).$$ 
It can be calculated that:
	\begin{align*}
		&kP_{DP}(K_{p - 1}\vee \mathcal{H},k - 1)+2(k - \text{col}(K_{p - 1}\vee \mathcal{H})-r)^{n - r-1 + p}\\
		&\geq k\left(\left(\prod_{j = 1}^{p - 1}(k - j)\right)P_{DP}(\mathcal{H},k - p)+f(k)\right)+2(k - \text{col}(K_{p - 1}\vee \mathcal{H})-r)^{n - r-1 + p}\\
		&=\left(\prod_{j = 0}^{p - 1}(k - j)\right)P_{DP}(\mathcal{H},k - p)+kf(k)+2(k - \text{col}(K_{p - 1}\vee \mathcal{H})-r)^{n -r-1 + p}.
	\end{align*}
    Let $g(k)=kf(k)+2(k - \text{col}(K_{p - 1}\vee \mathcal{H})-r)^{n - r-1 + p}$. Then $g(k)$ is a polynomial in $k$ of degree $n - r-1 + p$ with a leading coefficient of $2(p - 1)+2 = 2p$. Moreover, we obtain that
	\begin{align}\nonumber
		P_{DP}(K_{p}\vee \mathcal{H},k)\geq\min\left\{P(K_{p}\vee \mathcal{H},k),\left(\prod_{j = 0}^{p - 1}(k - j)\right)P_{DP}(\mathcal{H},k - p)+g(k)\right\}.
	\end{align}
		
\end{proof}

Now we are going to end this section with a partial answer to 
Problem \ref{Pro-3}. 

\begin{theorem}\label{Ans-3}
	Suppose that $\mathcal{H}$ is an r-uniform hypergraph on $n$ vertices such that
	$$
	P(\mathcal{H},k)-P_{DP}(\mathcal{H},k)=O(k^{n-r-1})\quad\text{as }k\to\infty.
	$$
	Then there exist $p\in \N$
	such that 
	$P_{DP}(K_p \vee \mathcal{H},k) 
	\infe P(K_p \vee \mathcal{H},k)$.
\end{theorem}

\begin{proof}
	Since the result is trivial when $\text{col}(\mathcal{H})\leq 2$, suppose that $\text{col}(\mathcal{H})\geq 3$. Additionally, there exist constants $C,N_1\in\mathbb{N}$ satisfying
	$$
	P(\mathcal{H},k)-P_{DP}(\mathcal{H},k)\leq Ck^{n - r-1}
	$$
	whenever $k\geq N_1$.
	
    Fix a natural number $p$ with $p > C/2$. Then for all $k\geq p + N_1$, we have that
	\begin{align*}
		\left(\prod_{j = 0}^{p - 1}(m - j)\right)P_{DP}(\mathcal{H},k - p)&\geq\left(\prod_{j = 0}^{p - 1}(k - j)\right)\left(P(\mathcal{H},k - p)-C(k - p)^{n - 3}\right)\\
		&=P(K_p \vee \mathcal{H},k)-C(k - p)^{n - r-1}\prod_{j = 0}^{p - 1}(k - j).
	\end{align*}
	Corollary \ref{co-2-1} implies that for $k\geq\text{col}(G)+r+ p$,
	\equ{eq6}
	{
	P_{DP}(K_p \vee \mathcal{H},k)\geq\min\left\{P(K_p \vee \mathcal{H},k),\left(\prod_{j = 0}^{p - 1}(k - j)\right)P_{DP}(\mathcal{H},k - p)+f(k)\right\}
}
	where $f(k)$ is a polynomial in $k$ of degree $n - r-1 + p$ with a leading coefficient of $2p$. 
	Because $p > C/2$, it follows that $f(k)-C(k - p)^{n - r-1}\prod_{j = 0}^{p - 1}(k - j)$ is a polynomial of degree $n - r-1 + p$ with a positive leading coefficient. Thus there exists an $N\in\mathbb{N}$ such that for each $k\geq N$	
\eqn{eq7}
{
P(K_p \vee \mathcal{H},k)&\leq& P(K_p \vee \mathcal{H},k)+f(k)-C(k - p)^{n - 3}\prod_{j = 0}^{p - 1}(k - j)\nonumber \\
		&\leq &\left(\prod_{j = 0}^{p - 1}(k - j)\right)P_{DP}(\mathcal{H},k - p)+f(k).
}	
	 The theorem then follows directly from (\ref{eq6})
	 and  (\ref{eq7}).
\end{proof}

\vskip 3mm
\noindent{\bf Acknowledgment  }

This research is supported by National Natural Science Foundation of
China under Grant
No. 12371340.

\end{document}